\documentclass{article}
\usepackage{latexsym}
\usepackage{graphicx}

\newtheorem{theorem}{Theorem}[section]

\newtheorem{conjecture}[theorem]{Conjecture}
\newtheorem{corollary}[theorem]{Corollary}
\newtheorem{observation}[theorem]{Observation}
\newtheorem{lemma}[theorem]{Lemma}

\newcommand{\heading}[1]{\vspace{1ex}\par\noindent{\bf #1}}

\newcommand\eps{\varepsilon}

\newcommand{\ProofEndBox}{{\ifhmode\unskip\nobreak\hfil\penalty50 \else
          \leavevmode\fi\quad\vadjust{}\nobreak\hfill$\Box$
            \finalhyphendemerits=0 \par}}%
\newcommand{\proofend}{\ProofEndBox\smallskip}

\newenvironment{proof}{\par\medskip\noindent{\bf Proof: }}
{\hskip 2cm\unskip\hbox{}\hfill$\Box$\par\bigskip}
\newenvironment{proofof}{\par\medskip\proofof}
{\unskip\hfill$\Box$\par\bigskip}
\def\proofof #1{\noindent{\bf Proof of #1:\hskip 0.5em}}

\def\ceil#1{\lceil #1 \rceil}
\newcommand\sep{:\,}
\def\:{\colon}

\title{Induced trees in triangle-free graphs}
\author{
{\sc Ji\v{r}\'{\i} Matou\v{s}ek} \ \ \ \ {\sc Robert \v{S}\'amal%
\thanks{Currently on leave from Institute for Theoretical
     Computer Science (ITI). The paper was finished while the second author was
     a PIMS postdoctoral fellow at Department of Mathematics, Simon Fraser
     University, Burnaby, B.C. V5A 1S6, Canada.}
   }\\
   {\footnotesize Department of Applied Mathematics and}\\[-1.5mm]
   {\footnotesize Institute of Theoretical Computer Science (ITI)}\\[-1.5mm]
   {\footnotesize  Charles University}\\[-1.5mm]
   {\footnotesize  Malostransk\'{e} n\'{a}m. 25, 118~00~~Praha~1}\\[-1.5mm]
   {\footnotesize Czech Republic}
}
\date{}

\begin{document}

\maketitle

\begin{abstract}
We prove that every connected triangle-free graph on $n$ vertices contains
an induced tree on $\exp(c\sqrt{\log n}\,)$ vertices, where $c$ is  
a positive constant. The best known upper bound is $(2+o(1))\sqrt n$. 
This partially answers questions of Erd\H{o}s, Saks, and S\'os and of Pultr.
\end{abstract}

\leftline{{\bfseries Keywords:\enspace}
  induced subgraphs, trees, Ramsey-type theorems
}

\leftline{{\bfseries MSC:\enspace}
05C55, 05C05 
}

\section{Introduction}\label{s:uvod}

For a graph $G$, let $t(G)$ denote the maximum number of vertices
of an induced subgraph of $G$ that is a tree (i.e., connected and
acyclic). 
There are arbitrary large graphs $G$ with $t(G) \le 2$, namely 
graphs in which every connected component is a clique.
To rule out these trivial examples, we need to put some restrictions
on~$G$. 

Motivated by study of forbidden configurations in Priestley 
spaces~\cite{BP-Priestley},
Pultr (private communication, 2002) asked how big $t(G)$ can be
if $G$ is connected and bipartite. Formally, he was interested
about asymptotic properties of the function
$$
  f_B(n)=\min\{t(G)\sep |V(G)|=n,\,\mbox{$G$ connected and bipartite}\}.
$$
Pultr's question was the starting point of our work.
However, the function $t(G)$ was studied earlier and in a more
general context by Erd\H{o}s, Saks, and S\'os \cite{ESS-itrees}. They
describe the influence of the number of edges of~$G$ on~$t(G)$ and, 
more to our point, they study how small $t(G)$ can be if $\omega(G)$ is given.
They observe that $t(G) \le 2 \alpha(G)$, and this
allows them to use estimates for Ramsey numbers. This way, they show
that for any fixed $k>3$ there are constants $c_1$, $c_2$ such that
$$
   c_1 \frac{\log n}{\log \log n} \le 
    \min \{ t(G) \sep |V(G)|=n, \, G \not\supseteq K_k \} \le c_2 \log n \,.
$$
For $k=3$ the lower bound still applies, but the upper bound obtained
by using Ramsey numbers was only $O(\sqrt n \log n)$ (nowadays this
approach yields $O(\sqrt{n\log n})$, due to the improved lower bound
on $R(k,3)$, see \cite{Kim}). 
We concentrate on this case $k=3$, that is we put
$$
f_T(n)=\min \{t(G)\sep |V(G)|=n,\,\mbox{$G$ connected and triangle-free}\}.
$$
Instead of applying Ramsey theory, we approach the problem directly.

It is easy to show that $f_T(n)\leq f_B(n)=O(\sqrt n\,)$. The best
construction we are aware of yields $f_B(n)\leq (2+o(1))\sqrt n$;
see~Section~\ref{s:init}. 
A simple ``blow-up'' construction, also presented in Section~\ref{s:init},
shows that if $f_T(n_0)<\sqrt {n_0}$ for some $n_0$, then
$f_T(n)=O(n^{1/2-\eps})$ for a positive constant $\eps>0$, and similarly
for $f_B$.
Hence, $f_T(n)$ either is of order exactly $\sqrt n$, or it is bounded above by
some power strictly smaller than $1/2$. We conjecture that the
second possibility holds, and that another power of~$n$ is a 
lower bound. 

\begin{conjecture}\label{c:b}
There are constants $0 < \alpha < \beta < 1/2$, and $c_1$, $c_2$ 
such that for all $n$
$$
      c_1 n^\alpha \leq  f_T(n) \leq f_B(n) \leq c_2 n^\beta \,.
$$
\end{conjecture}

The following lower bound is the main result of this paper.

\begin{theorem}\label{t:t}
There is a constant $c>0$ such that for all $n$
$$
f_T(n)\geq e^{c\,\sqrt{\log n}} \,.
$$
\end{theorem}

We finish the introduction by mentioning further
results concerning~$t(G)$. It is interesting to consider the problem 
of finding induced trees in (sparse) random graphs.
Vega~\cite{Vega-itree} shows that $t(G_{n,c/n}) = \Omega(n)$ a.s.;
Palka and Ruci{\'n}ski~\cite{PR-itree} prove
that $t(G_{n,c\log n/n}) = \Theta(n \log \log n/\log n)$ a.s.

Krishnan and Ochem~\cite{KO-itrees} search for values of $f_T(n)$
(for small~$n$) using a computer; they succeed to find $f_T(n)$ for $n \le 15$.
They also extend results of~\cite{ESS-itrees} about the 
decision problem: ``given a connected graph~$G$ and an integer~$t$,
does~$G$ have an induced tree with~$t$ vertices?''. Not only this is
NP-complete for general graphs (which is proved in~\cite{ESS-itrees}), 
but it remains NP-complete even if we restrict to bipartite graphs, 
or to triangle-free graphs of maximum degree~4.

\section{Initial observations}\label{s:init}

\begin {observation}   \label{l:path}
$f_B(n) \le (2+o(1))\sqrt n$.
\end {observation}

\begin {proof}
It is enough to take a path with each edge replaced
by a complete bipartite graph.
More precisely, we take pairwise disjoint sets $V_i$ 
(for $i = -(k-1), \dots, k-1$) such that $|V_i|=k-|i|$.
We let $G$ be the graph with vertices $V=\bigcup_{|i|<k} V_i$
and all possible edges between $V_i$ and~$V_{i+1}$ (for $i=-(k-1), \dots, k-2$).

It is clear that if an induced tree in~$G$ contains
a vertex from $V_i$ and two vertices from $V_{i+1}$
then it contains no vertex of $V_j$ for $j > i+1$; similarly
for $i+1$ replaced by $i-1$. Therefore any maximum induced tree
is one of trees $T_{a,b}$ ($-(k-1) < a < b < k-1$ and $b-a > 1$):
it contains all vertices from two levels, $V_a$ and $V_b$
and one vertex from each $V_i$ where $a < i < b$.
It is easy to compute that such tree contains $2k-1$ vertices
out of the $|V|=k^2$; this proves $f_B(k^2) \le 2k-1$.
If $(k-1)^2 < n \le k^2$ then we take a subgraph of~$G$ 
to show that $f_B(n) \le 2k-1 < 2 \sqrt n + 1$.
\end {proof}

\begin{lemma}[Blow-up construction]\label{l:blowup}
Let $G$ be a connected triangle-free graph and let $W\subseteq V(G)$
be a subset of $m$ vertices ($m\ge 3$) such that any induced tree in~$G$
contains at most~$t$ vertices of $W$. Then we have
$f_T(n)=O(n^{\ln(t-1)/\ln(m-1)})$. The same result holds
with ``triangle-free'' replaced by ``bipartite'' and with $f_T$
replaced by $f_B$.
\end{lemma}

\begin {proof}
We let $W = \{w_0, \dots, w_{m-1}\}$, and
write $r=m-1$ and $q=t-1$ to simplify expressions.
As $G$ is triangle-free it follows that $t\ge 3$, and so $q\ge 2$.

Let $T = T_{r,l}$ be a rooted tree with $l$ levels in which each non-leaf
vertex has $r$ sons. Next, for each vertex $v$ of~$T$ we take
a copy $G_v$ of~$G$ (so that distinct copies are disjoint).
Whenever $v$ is a non-leaf vertex of~$T$ and $u$ is its $i$-th son, 
we introduce an edge between $w_i$ in $G_v$ and $w_0$ in~$G_u$;
the resulting graph will be called $T(G)$ (see Fig.~\ref{fig:construction}).
Clearly this graph
is triangle-free/bipartite if $G$ was triangle-free/bipartite.
Moreover, $|V(T(G))| = |V(T)| \cdot |V(G)|$ and 
$|V(T)| = \frac { r^{l+1}-1 }{ r-1 } = \Theta(r^l)$
(since $l \to \infty$ and $r \ge 2$). 

\def\Sbar{\bar S}
Let $S$ be an induced subtree of~$T(G)$ and put 
$$
  \Sbar = \{v \in V(T) \mid \mbox{$G_v$ contains a vertex of~$S$} \} \,.
$$
By construction, $S \cap G_v$ is a tree in $G_v$ for each~$v$.
So the condition on~$G$ implies that each vertex of~$\Sbar$ has at most~$t$ neighbors
in~$\Sbar$. Consequently, we have (since $q\ge 2$)
$$
  |\Sbar| \le 1 + \sum_{i=1}^l (q+1) q^{i-1} 
          \le 1 + (q+1) \frac { q^l-1 }{ q-1 } 
          = \Theta (q^l) \,.
$$

Now recall that $q$, $r$, and $|V(G)|$ are constants.
For a given $n$, choose the smallest~$l$ such that $n \le |V(T_{r,l}(G))|$; 
we have $n = \Theta(r^l)$. By the above considerations, 
$$
 f(n) \le f(T_{r,l}(G)) 
      \le |V(G)| \cdot \Theta(q^l)
        = \Theta (r^{l \log_r q})
        = \Theta (n^{\log_r q})  \,,
$$
which finishes the proof.
\end {proof}

\begin {corollary}   \label{sqrt}
If $f_T(n_0) < \sqrt {n_0}$ for some $n_0$,
then $f_T(n) = O(n^{1/2-\eps})$ for a positive constant $\eps>0$.
(The same is true for $f_B$.)
\end {corollary}

\begin {proof}
Let $G$ be the graph on $n_0$ vertices for which $t(G) = t < \sqrt{n_0}$.
We let $W = V(G)$ and $m=n_0$ and apply Lemma~\ref{l:blowup}.
\end {proof}

As mentioned in the introduction, Krishnan and Ochem~\cite{KO-itrees} search for
values of $f_T(n)$ using a computer. This was motivated by hope that
Corollary~\ref{sqrt} would apply. It turns out, however, that for small $n$
Observation~\ref{l:path} gives a precise estimate even for $f_T(n)$ (e.g.,
$f_T(15)=7$); therefore Corollary~\ref{sqrt} does not apply.

\heading{Remark. }
If we consider the construction from Lemma~\ref{l:blowup}
for $G = K_3$, $W = V(G)$, $m=3$, and $t=2$ we recover
a result of~\cite{ESS-itrees} that there is a graph~$G$ containing
triangles (but no~$K_4$) such that $t(G) = O(\log n)$.

\begin {figure}
  \includegraphics{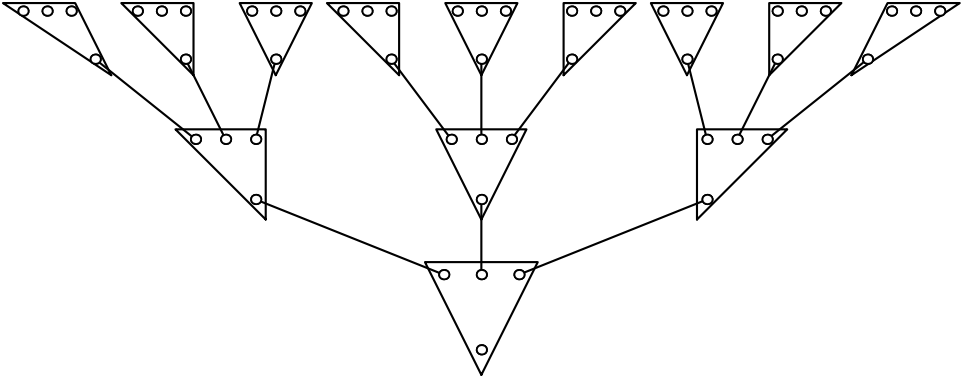}
  \caption{Graph $T_{3,2}(G)$ from the proof of Lemma~\ref{l:blowup}.} 
  \label{fig:construction}
\end {figure}

\section{Lower bound for bipartite graphs}\label{s:weak}

Here we prove a statement weaker than Theorem~\ref{t:t}---we 
give a bound on $f_B(n)$ instead of $f_T(n)$.
The proof is simpler than that of Theorem~\ref{t:t}
and it serves as an introduction to it.

We begin with a lemma about selecting induced forests of a particular kind in
a bipartite graph.
We introduce some terminology. Let $H$ be a bipartite graph
with color classes $A$ and $B$. We will think of $A$ as the ``top''
class and $B$ as the ``bottom'' class (in a drawing of $G$ in the plane, say). 
We write $a=|A|$ and $b=|B|$. For a subgraph $F$ of $H$ we write
$A(F)=V(F)\cap A$,
we set $a(F)=|A(F)|$, and we define $B(F)$ and $b(F)$ similarly.

Whenever we say {\em forest\/}
we actually mean an induced subgraph of $H$ that is a forest.
An {\em up-forest\/} $F$ is a forest such that every vertex
in $A(F)$ has degree (in~$F$) precisely 1 and every vertex in $B(F)$ has
degree (in $F$) at least 1.

\begin {figure}[ht]
  \centerline{\includegraphics{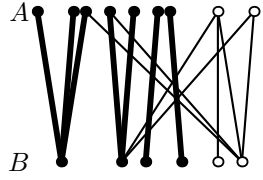}}
  \caption{An up-forest}
  \label{fig:upforest}
\end {figure}

A {\em matching\/} is a forest~$F$ in which all vertices have degrees (in $F$) exactly 1.

\begin{lemma}\label{l:slabe}
Let $H$ be a bipartite graph with color classes $A$ and $B$ as above,
let $\Delta$ be the maximum degree of $H$, and let $\eta\in (0,1)$
be a real parameter. Let us suppose that every vertex in $A$
is connected to at least one vertex in $B$. Then at least one
of the following cases occurs:
\begin{enumerate}
\item[\rm(M)] There is a matching with at least $(1-\eta)a$ edges.
\item[\rm(B)] There is an up-forest $F$ with
$$
b(F)\geq \frac\eta{\Delta^3}\cdot a
$$
that is $2$-branching, meaning that
every vertex  in $B(F)$ has degree at least~$2$ in $F$.
\end{enumerate}
\end{lemma}

\begin {figure}[ht]
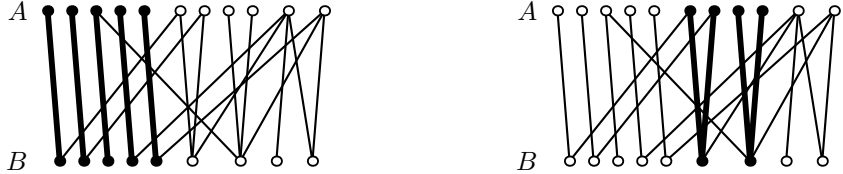

  \centerline{\hfill\includegraphics{itree-fig.3}\hfill\includegraphics{itree-fig.4}\hfill}
  \caption{An illustration of Lemma~\ref{l:slabe}}
  \label{fig:illulemma}
\end {figure}

\heading{Proof. } Let $B'\subseteq B$ be the set of vertices of degree~1
in $B$. If $|B'|\geq (1-\eta)a$ then, clearly, case (M) occurs,
so we may assume $|B'|< (1-\eta)a$. Let $B''\subseteq B$ consist
of all vertices of degree at least 2. Since every vertex
in $A$ has degree at least 1, $|E(H \setminus N(B'))|\geq \eta a$,
and so $|B''|\geq (\eta/\Delta)a$.

Let us set $B_0=B''$ and let $F_0$ be an empty graph.
Supposing that a set $B_{i-1}\subseteq B''$ and an up-forest $F_{i-1}$
have already been constructed with $B_{i-1}\neq\emptyset$, 
we construct $B_{i}$ and $F_{i}$.
We let $v_i$ be an arbitrary vertex in $B_{i-1}$, and we let $S_i$
be the star formed by $v_i$ and all of its neighbors in $A$.
We set $F_{i}=F_{i-1}\cup S_i$, we let $N_i\subseteq B$ be the
neighborhood of $A(S_i)$,
and we let $B_i$ be $B_{i-1}\setminus N_i$.
The construction finishes when $B_i=\emptyset$, with $F_{i}$
as the resulting up-forest.

It is easy to check that this construction indeed yields an up-forest $F$
with each degree in $B(F)$ at least 2.
We have $a(S_i)\leq \Delta$ and $|N_i|\leq a(S_i)(\Delta-1)+1$, and so
in each step, at most $|N_i|\leq \Delta(\Delta-1)+1\leq\Delta^2$
vertices are removed from $B_i$. Having started with at least
$(\eta/\Delta)a$ vertices, we can proceed for at least
$(\eta/\Delta^3)a$ steps, and so the resulting up-forest is as in (B).
\proofend
\medskip

Now we prove the lower bound
$$
f_B(n)\geq e^{c\sqrt{\log n}}
$$
for a constant $c>0$.

Let $G$ be a given connected bipartite graph. 
We assume that $n=|V(G)|$ is sufficiently large whenever convenient.
We let $t$ be the
``target size'' of an induced tree in $G$ we are looking for;
namely, $t=\ceil{\exp(c\sqrt{\log n}\,)}$.
If $G$ has a vertex of degree at least $t-1$, then we can take its star for the
induced tree and we are done, so we may assume that the maximum degree
satisfies $\Delta \le t-2$.

Let us fix an arbitrary vertex of $G$ as a root, and let $L_i$ be the
set of vertices of~$G$ at distance precisely $i$ from the root.
All edges of $G$ go between $L_{i-1}$ and $L_{i}$ for some $i$, since 
an edge within some $L_i$ would close an odd cycle.

We may assume that $L_t=\emptyset$,
for otherwise $G$ contains an induced path of length $t$.
Hence there is a $k$ with $|L_k|\geq n/t$. 

Let us fix such a $k$. We are going to construct
sets $M_{i}\subseteq L_i$, $i=k,k-1,\ldots$, inductively, until
we first reach an $i$ with $|M_i|=1$ (this happens for $i=0$ at the latest
since $|L_0|=1$). We shall let $\ell$ be this last $i$.

Suppose that nonempty sets $M_k,M_{k-1},\ldots,M_{i}$
have already been constructed,
in  such a way that the subgraph of $G$ induced 
by $M_k\cup\cdots\cup M_{i}$ is a forest, each of whose 
 components intersects $M_{i}$ in at most one
vertex. We are going to construct $M_{i-1}$.

Let us put $A=M_{i}$, $B=L_{i-1}$, and let us consider the
bipartite graph $H$ induced by $A \cup B$ in~$G$.
Every vertex of $A$ is connected to at least
one vertex in $B$. We set $\eta=\frac 1t$ 
and apply Lemma~\ref{l:slabe}.
This yields an up-forest $F$ in $H$ as in the lemma.
We define $M_{i-1}=B(F)$. 

If $F$ is a matching, i.e., case (M) occurred in the lemma, 
we call the step from $M_{i}$ to  $M_{i-1}$
a {\em matching step}. In this case, we have $|M_{i-1}|\geq
(1-\frac 1t)|M_{i}|$. Otherwise, $F$ is a 2-branching 
forest; then we call the step a {\em branching step}
and we have $|M_{i-1}|\geq |M_{i}|/(t\Delta^3) \geq
|M_{i}|/t^4$.

Suppose that the sets $M_k,\ldots,M_{\ell}$
have been constructed, $|M_{\ell}|=1$.
We claim that the number $b$ of branching steps in the construction is at least
$c_1\sqrt{\log n}$ for a suitable constant $c_1>0$.
Indeed, there are no more than $t$ matching steps,
and so $1=|M_{\ell}|\geq |M_k|(1-1/t)^t t^{-4b}
\geq (n/t)e^{-1}/2 \cdot t^{-4b} = \Omega(nt^{-4b-1})$.
Thus $b=\Omega(\log n/\log t)=\Omega(\sqrt{\log n}\,)$,
since $t=\ceil{\exp(c\sqrt{\log n}\,)}$.

It is easy to see that $M_k\cup M_{k-1}\cup\cdots\cup M_{\ell}$
induces a forest in $G$. We let $T$ be the component of this forest
containing the single vertex of $M_{\ell}$. Since every vertex
of $M_{i-1}$, $\ell < i\leq  k$, has at least one neighbor in $M_{i}$,
and if the step from $M_i$ to $M_{i-1}$ was a branching step then
each vertex of $M_{i-1}$ has at least two neighbors in $M_{i}$,
it follows that $T$ has at least $2^b=\exp(\Omega(\sqrt{\log n}\,))$
vertices. This finishes the proof of the lower bound
$f_B(n)\geq \exp(c\sqrt{\log n}\,)$.
\proofend

\heading{Remark. } The above proof may seem wasteful in many respects.
However, the result is tight up to the value of the constant in the exponent
if we insist on selecting an induced tree ``growing up'' 
(i.e., made of up-forests for some choice of root and corresponding sets~$L_i$). 
Indeed, any such induced tree in the graph $G_r$
in Figure~\ref{fig:qt} may contain at most two of the $r$ vertices at the topmost
level of the graph. 
Let us put $r=\exp(c\sqrt{\log n}\,)$ and glue copies of
$G_r$ according to the pattern of a complete $r$-ary tree (as in the proof
of Lemma~\ref{l:blowup}), so that the resulting graph has approximately $n$
vertices (that is, the depth is $l = \Theta(\log n)$.
We obtain a graph with all up-growing induced trees having size
at most $2^l = \exp(O(\sqrt{\log n}\,))$.

\begin {figure}[ht]
  \centerline{\includegraphics{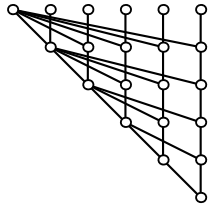}}
  \caption{Graph $G_6$ in which all ``up-growing trees'' contain at most
    two vertices of the uppermost level.}
  \label{fig:qt}
\end {figure}

\section{Lower bound for triangle-free graphs}\label{s:tf}

Here we prove Theorem~\ref{t:t}. The scheme of the proof is very similar to
the proof of the same bound for bipartite graphs in Section~\ref{s:weak}.
We continue using the definitions and notation from that proof.
So we decompose the given graph into the levels $L_0,L_1,\ldots,L_r$,
$r<t$. The main difference compared to the bipartite case is that
there may now be edges within the levels $L_i$. We will need the well-known
fact that any graph on $n$ vertices with maximum degree $\Delta$ contains
an independent set of size at least $n/(\Delta+1)$.
We will also need the following simple modification.

\begin{lemma}\label{l:turlike}
Let $\Gamma$ be a graph (not necessarily bipartite)
on $n$ vertices with maximum degree $\Delta$, and let $\eta\in [0,1]$
be a real parameter. Then at least one of the following two cases occurs:
\begin{enumerate}
\item[\rm(IS)] $\Gamma$ contains an independent set with at least $(1-\eta)n$ vertices.
\item[\rm(IM)] $\Gamma$ contains an induced matching with at least $\frac{\eta}{2\Delta} n$
edges.
\end{enumerate}
\end{lemma}
\heading{Proof. } We repeatedly select edges $e_1,e_2,\ldots$ of $\Gamma$;
having selected $e_i$, we delete it and all the neighbors of its endvertices from
the current graph. In each step we delete at most $2\Delta$ vertices,
so we either construct an induced matching as in (IM) or reach an edgeless graph
after deleting at most $\eta n$ vertices, hence yielding an induced set
as in (IS).
\proofend

\heading{Proof of Theorem~\ref{t:t}. }
We proceed similarly as in the previous section.
We suppose $G$ is a given triangle-free graph on $n$ vertices
(and that $n$ is big enough), we put $t=\ceil{\exp(c\sqrt{\log n}\,)}$.
Again, we may assume $t\le \Delta-2$: $G$~is triangle-free, so a star
of a vertex is an induced tree.

As before, we begin by selecting a root vertex and constructing the at most~$t$ 
levels $L_0$, $L_1$, \ldots. We select~$k$ such that $|L_k|\geq n/t$ and we will
construct sets $M_k$, $M_{k-1}$, \dots, $M_{\ell}$, such that $M_i\subseteq L_i$,
$|M_{\ell}|=1$, in such a way that their union induces a forest in~$G$. In the
induction step, we will either construct $M_{i-1}$ from $M_i$, or sometimes we
will go down two levels at once, producing both $M_{i-1}$ and $M_{i-2}$.

We begin by selecting $M_k$ as an independent set
in the subgraph induced by $L_k$.
By the fact mentioned before Lemma~\ref{l:turlike}
we may assume $|M_k|\geq |L_k|/t\geq n/t^2$.

We suppose that $M_i$ has already been constructed so that  
each component of the forest induced by $M_k\cup\cdots\cup M_i$
intersects $L_i$ in at most one vertex (and, in particular, $M_i$
is an independent set). Now we proceed as
in the proof in Section~\ref{s:weak}: 
We let $A=M_i$, $B=L_{i-1}$, and we consider
the bipartite graph $H$ induced by~$A \cup B$ in~$G$.
We apply Lemma~\ref{l:slabe}
to $H$ with $\eta=\frac 1t$, obtaining an up-forest $F$. 
We set $M'_{i-1}=B(F)$; this is not yet the final $M_i$ 
since there may be edges on $M'_{i-1}$.

If case (B) occurred in Lemma~\ref{l:slabe}, 
we have $|M'_{i-1}|\geq |M_i|/t^4$.
We let $M_{i-1}$ be an independent set of 
size $|M'_i|/(\Delta+1)\geq |M_i|/t^5$ 
in the subgraph induced by $M'_{i-1}$.
We call this step a {\em branching step}.

If case (M) occurred in Lemma~\ref{l:slabe}, 
we have $|M'_{i-1}|\geq (1-\frac 1t)|M_i|$.
Then we apply Lemma~\ref{l:turlike} with $\eta=\frac 1t$
to the graph $\Gamma$ induced in $G$ by $M'_{i-1}$.
If case (IS) applies in that lemma, we let $M_{i-1}$ be the independent
set of size at least $(1-\frac 1t)|M'_{i-1}|\geq (1-\frac1t)^2|M_i|$;
we call this step a {\em matching step}. Both the matching step
and the branching step go one level down, from $i$ to $i-1$.

If case (IM) applies in Lemma~\ref{l:turlike},
we define $M_{i-1}$ as the vertex 
set of the induced matching from the lemma.
In this case we have $|M_{i-1}|\geq (1-\frac 1t)|M_i|/t^2$.
Note that this $M_{i-1}$ does not satisfy the inductive assumption
(it is not an independent set). We are also going to 
construct $M_{i-2}$ in the same step, thus going from $i$ to $i-2$.
To obtain  $M_{i-2}$, we define another
auxiliary bipartite graph, which we again call  $H$ to save letters.
The bottom color class $B$ is $L_{i-2}$, and the top color class
$A$ is obtained by contracting the edges induced by $M_{i-1}$.
More formally, we set $A=\{ uu'\in E(G)\sep u,u'\in M_{i-2}\}$, 
$B=L_{i-2}$, and 
$E(H)=\bigl\{ \{ uu', v\}  \sep u,u'\in A, v\in B,\, uv\in E(G)\mbox{ or } u'v\in E(G)\bigr\}$. 
(Note that in this definition it can not happen that both $uv$ and $u'v$ are edges of~$G$, 
as $G$ is triangle-free.)
We apply Lemma~\ref{l:slabe} with $\eta=\frac12$, say, to $H$.
In both of cases (M) and (B) we obtain an up-forest $F$ in $H$
with $b(F)\geq |M_{i-1}|/(32t^3)$ (we note that $|A|=\frac 12|M_{i-1}|$
and that $H$ has maximum degree no larger than $2t$).
We set $M'_{i-2}=B(F)$, and finally, we select $M_{i-2}$ as an
independent set of size at least $|M'_{i-2}|/t$ in the subgraph
induced by $M'_{i-2}$.  Since $G$ is triangle-free, one can check
that $M_k\cup\cdots\cup M_{i-1}\cup M_{i-2}$ induces a forest. 
We have $|M_{i-2}|\geq
|M_{i-1}|/32t^4\geq |M_{i}|\cdot(1-\frac 1t)/32t^6\geq |M_i|/t^7$.
We call this step from $M_{i}$ to $M_{i-2}$
a {\em double-step}. 

By calculation similar to that in Section~\ref{s:weak}, we find that
the number $b$ of branching steps and double-steps together is
at least $\Omega(\sqrt{\log n}\,)$. We again claim that the
component of the forest induced by  $M_{k}\cup\cdots \cup M_{\ell}$
containing the single vertex of $M_{\ell}$ has at least $2^b$ vertices.
Indeed, if $M_i$ was obtained from $M_{i+1}$ by a branching step,
then each vertex of $M_i$ has at least two successors in $M_{i+1}$.
If $M_i$ was obtained from $M_{i+2}$ by a double-step, then
each vertex $v$ of $M_i$ has at least one succesor in $M_{i+1}$, this
is connected by an edge to precisely one other vertex of $M_{i+1}$, 
and both of these vertices have one neighbor in $M_{i+2}$;
consequently $v$ has at least two successors in $M_{i+2}$.
Theorem~\ref{t:t} is proved.\proofend

\subsection*{Acknowledgment}
We would like to thank participants of a research seminar where initial
steps of this work have been made, for a stimulating environment.
In particular, we are indebted to Robert Babilon, Martin B\'alek, 
Helena Nyklov\'a, Ondra Pangr\'ac, and Pavel Valtr 
for useful discussions and observations.

\bibliographystyle{rs-amsplain}
\bibliography{itree}

\end{document}